\documentclass[12pt]{amsart}
\pdfoutput=1

\usepackage{lmodern} 
\usepackage[scale=0.89]{tgheros} 
\usepackage[osf]{Baskervaldx} 
\usepackage[baskervaldx,cmintegrals,bigdelims,vvarbb]{newtxmath} 
\usepackage[cal=boondoxo]{mathalfa} 

  \usepackage{amssymb}         
  \usepackage{amsmath}          
  \usepackage{amsfonts}           
  \usepackage{amsthm}
  \usepackage{amscd}
  \usepackage{enumerate}
  \usepackage[active]{srcltx}
  \usepackage[percent]{overpic}
  \usepackage[mathscr]{eucal}
  \usepackage{graphicx,epsfig} 
  \usepackage{tikz-cd}        
  \usepackage{color}
  \usepackage{marginnote}
  \usepackage{hyperref}
  \usepackage{enumitem} 
  \usepackage{bbm}

  \newcommand*{\e}{\ensuremath{{\operatorname{e}}}}
  
  \def\om{\omega}
  
  \def\La{\Lambda}
  \def\si{\sigma}  
  \def\RR{{\mathbb{R}}}

  \def\TT{{\mathbb{T}}}
  \def\whd{\hat{d}}
  \def\whD{\hat{D}}  
  \def\ZZ{{\mathbb{Z}}}

\newtheorem{theorem}{Theorem}[section]
\newtheorem{corollary}[theorem]{Corollary}
\newtheorem{lemma}[theorem]{Lemma}
\newtheorem{proposition}[theorem]{Proposition}

\theoremstyle{definition}
\newtheorem{definition}[theorem]{Definition}
\newtheorem{example}[theorem]{Example}
\theoremstyle{remark}
\newtheorem{remark}[theorem]{Remark}
\newtheorem{conjecture}{Conjecture}
\theoremstyle{thmx}

\newcommand{\ALIGN}{\begin{align*}}
\newcommand{\ENDALIGN}{\end{align*}}
\newcommand{\ENUM}{\begin{enumerate}}
\newcommand{\ENUMa}{\begin{enumerate}[a.]}
\newcommand{\ENUMA}{\begin{enumerate}[A.]}
\newcommand{\ENUMAF}{\begin{enumerate}[\bf A.]}
\newcommand{\ENUMi}{\begin{enumerate}[i)]}
\newcommand{\ENDENUM}{\end{enumerate}}
\newcommand{\ITMZ}{\begin{itemize}}
\newcommand{\ENDITMZ}{\end{itemize}}
\newcommand{\REFEQN}[1] { \begin{equation}\label{#1} }
\newcommand{\ENDEQN}{\end{equation}}
\newcommand{\THM}{\begin{theorem}}
\newcommand{\REFEXA}[1] { \begin{example}\label{#1} }
\newcommand{\ENDEXA}{\end{example}}
\newcommand{\MTX}{ \begin{matrix}}
\newcommand{\ENDMTX}{ \end{matrix}}
\newcommand{\REM}{ \begin{remark}}
\newcommand{\ENDREM}{\end{remark}}
\newcommand{\REFTHM}[1] { \begin{theorem}\label{#1} }
\newcommand{\RREFTHM}[2] { \begin{theorem}[#1]\label{#2} }
\newcommand{\ENDTHM}{\end{theorem}}
\newcommand{\REFNTH}[1] { \begin{newthm}\label{#1} }
\newcommand{\ENDNTH}{\end{newthm}}
\newcommand{\REFPROP}[1]{\begin{proposition}\label{#1} }
\newcommand{\RREFPROP}[2]{\begin{proposition}[#1]\label{#2} }
\newcommand{\PROP}{\begin{proposition}}
\newcommand{\ENDPROP}{\end{proposition} }
\newcommand{\REFDEF}[1]{\begin{definition}\label{#1} }
\newcommand{\DEF}{\begin{definition}}
\newcommand{\ENDDEF}{\end{definition} }
\newcommand{\REFLEM}[1]{\begin{lemma}\label{#1} }
\newcommand{\RREFLEM}[2]{\begin{lemma}[#1]\label{#2} }
\newcommand{\LEM}{\begin{lemma}}
\newcommand{\ENDLEM}{\end{lemma} }
\newcommand{\REFCOR}[1]{\begin{corollary}\label{#1} }
\newcommand{\COR}{\begin{corollary}}
\newcommand{\ENDCOR}{\end{corollary} }
\newcommand{\CONJ}{\begin{conjecture}}
\newcommand{\REFCONJ}[1]{\begin{conjecture}\label{#1}}
\newcommand{\RREFCONJ}[2]{\begin{conjecture}{#1}\label{#2}}
\newcommand{\ENDCONJ}{\end{conjecture} }
\newcommand{\REFDEFTHM}[1] { \begin{defthm}\label{#1} }
\newcommand{\ENDDEFTHM}{\end{defthm}}

\newcommand{\figref}[1]{Fig.~\ref{#1}}
\newcommand{\lemref}[1]{Lemma~\ref{#1}}

\newcommand{\thmref}[1]{Theorem~\ref{#1}}

\newcommand{\PROOF}{\begin{proof}}
\newcommand{\ENDPROOF}{\end{proof}}

\numberwithin{equation}{section}
\newcommand{\vs}{\vspace{6pt}}
\newcommand{\bit}{\it \bfseries}
\newcommand{\es}{\emptyset}
\newcommand{\sm}{\smallsetminus}
\newcommand{\ov}{\overline}
\newcommand{\bd}{\partial}
\newcommand{\ve}{\varepsilon}
\newcommand{\modd}{\ (\operatorname{mod} \ZZ)}
\newcommand{\ord}{\operatorname{ord}}
\def\CC{\mathbb{C}}

\begin{document}

\title{Periodic Points and Smooth Rays}

\author[C. L. Petersen and S. Zakeri]{Carsten Lunde Petersen and Saeed Zakeri}

\address{Department of Mathematics, Roskilde University, DK-4000 Roskilde, Denmark} 

\email{lunde@ruc.dk}

\address{Department of Mathematics, Queens College of CUNY, 65-30 Kissena Blvd., Queens, New York 11367, USA and The Graduate Center of CUNY, 365 Fifth Ave., New York, NY 10016, USA}

\email{saeed.zakeri@qc.cuny.edu}

\date{September 2, 2020}

\begin{abstract}
Let $P: \CC \to \CC$ be a polynomial map with disconnected filled Julia set $K_P$ and let $z_0$ be a repelling or parabolic periodic point of $P$. We show that if the connected component of $K_P$ containing $z_0$ is non-degenerate, then $z_0$ is the landing point of at least one {\it smooth} external ray. The statement is optimal in the sense that all but one ray landing at $z_0$ may be broken.     
\end{abstract}

\maketitle

\section{Introduction}\label{sec:intro}

It has been known since the pioneering work of Douady and Hubbard in the early 1980's that a repelling or parabolic periodic point of a polynomial map with connected Julia set is the landing point of finitely many external rays of a common period and combinatorial rotation number \cite{H,M,P}. When suitably formulated, this statement remains true for polynomials with disconnected Julia set although in this case one must also allow broken external rays that crash into the escaping (pre)critical points \cite{LP}. To state this precisely, let $P:\CC \to \CC$ be a polynomial of degree $D \geq 2$ with disconnected filled Julia set $K_P$. The external rays of $P$ can be defined in terms of the gradient flow of the Green's function $G$ of $P$ in $\CC \sm K_P$. They consist of smooth field lines that descend from $\infty$ and approach $K_P$, as well as their limits which are broken rays that abruptly turn when they crash into a critical point of $G$. For each $\theta \in \TT :=\RR/\ZZ$ the polynomial $P$ has either a smooth ray $R_\theta$ or a pair $R^\pm_{\theta}$ of broken rays which descend from $\infty$ at the normalized angle $\theta$. Here $R^+_{\theta}$ (resp. $R^-_{\theta}$) makes a right (resp. left) turn at each critical point of $G$ it crashes into (see \S \ref{sec:prelim} for details). \vs

Now suppose $z_0$ is a repelling or parabolic periodic point of $P$ with period $k$. Denote by $\La(z_0)$ the set of angles $\theta \in \TT$ for which $R_\theta$ or one of $R^\pm_{\theta}$ lands at $z_0$. Then $\La(z_0)$ is a non-empty ``rotation set'' under the $k$-th iterate of the map $\whD: \theta \mapsto D\theta \modd$. Moreover, the following dichotomy holds depending on the rotation number $\rho$ of the restriction $\whD^{\circ k}|_{\La(z_0)}$: Either $\rho$ is rational in which case $\La(z_0)$ is a union of finitely many cycles of the same length, or $\rho$ is irrational in which case $\La(z_0)$ is a minimal Cantor set. The first alternative is guaranteed to happen if the connected component of $K_P$ containing $z_0$ is non-degenerate (see \cite{LP} and compare \cite{PZ}). \vs

This note will sharpen the above statement by proving the following \vs  

\noindent
{\bf Main Theorem.} {\it Let $P: \CC \to \CC$ be a polynomial map with disconnected filled Julia set $K_P$ and let $z_0$ be a repelling or parabolic periodic point of $P$. If the connected component $K$ of $K_P$ containing $z_0$ is non-degenerate, then $z_0$ is the landing point of at least one {\em smooth} external ray.} \vs 

To reiterate, here is the equivalent (and a bit curious-sounding) formulation:  If every external ray landing on a repelling or parabolic point $z_0$ is broken, then the singleton $\{ z_0 \}$ is a connected component of the filled Julia set. \vs 

The non-degeneracy assumption on the component $K$ is essential. For example, the quadratic map $P(z)=z^2+c$ with real $c>1/4$ has a pair of complex conjugate fixed points $z^\pm_0$, and the components $K^\pm$ of $K_P$ containing $z^\pm_0$ are $\{ z^\pm_0 \}$ since $K_P$ is a Cantor set. If we label these fixed points so that $\operatorname{Im}(z^+_0)>0$ and $\operatorname{Im}(z^-_0)<0$, then $z^\pm_0$ is the landing point of the unique infinitely broken ray $R^\pm_0$ \cite{GM}. Note however that the Main Theorem holds trivially if $K= \{ z_0 \}$ but $\La(z_0)$ is infinite, for in this case $\La(z_0)$ is a Cantor set and since there are countably many broken rays, most rays landing at $z_0$ must be smooth. \vs

The assertion of the Main Theorem is optimal in that all but one ray landing at $z_0$ may in fact be broken. For example, \figref{seh} illustrates the disconnected filled Julia set of a cubic polynomial with a repelling fixed point $z_0$ at which the rays $R_0, R^+_{1/2}$ co-land. More generally, using the technique of quasiconformal surgery, one can construct polynomials of any degree $D \geq 3$, with disconnected filled Julia set, for which the $D-1$ fixed rays of $\whD$ co-land at a repelling fixed point in such a way that $D-2$ of these rays are broken and only one is smooth \cite[Theorem D and \S 7]{PZ}. 

\begin{figure}[t!]
	\centering
	\begin{overpic}[width=0.6\textwidth]{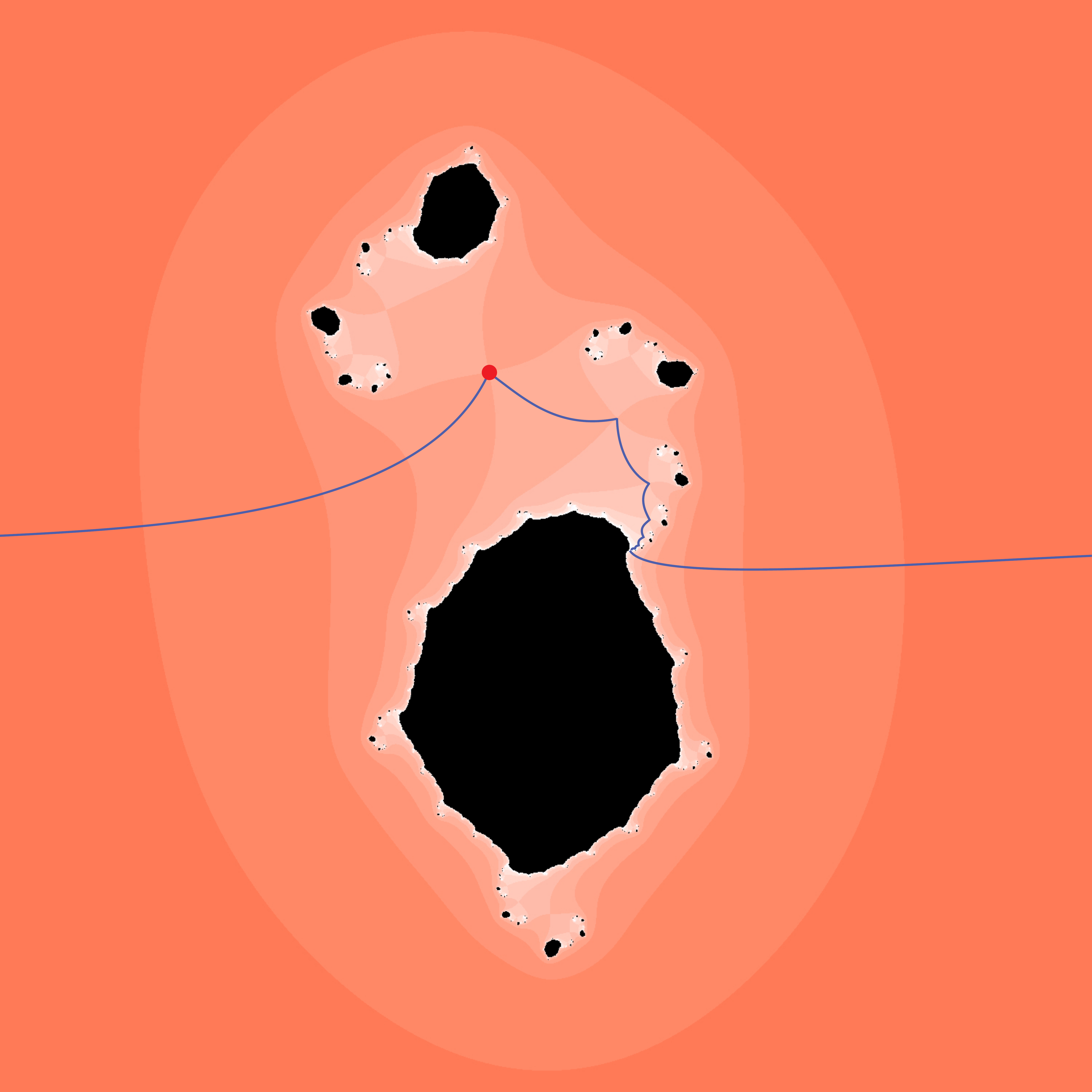}
		\put (94,51) {\small $R_0$}
		\put (1,47) {\small $R^+_{1/2}$}
		\put (53.5,47) {\small {\color{white} $z_0$}}
		\put (44.5,67.5) {\small $\omega$}
		\put (48,36) {{\color{white} $K$}}
	\end{overpic}
	\caption{\sl{Filled Julia set of the cubic polynomial $P(z)=a z^2+z^3$ with $a \approx 0.31629-i 1.92522$. The fixed rays $R_0, R^+_{1/2}$ co-land at the fixed point $z_0$. Here $R_0$ is smooth but $R^+_{1/2}$ is infinitely broken at the escaping critical point $\omega$ and its preimages.}} 
	\label{seh}   
\end{figure}
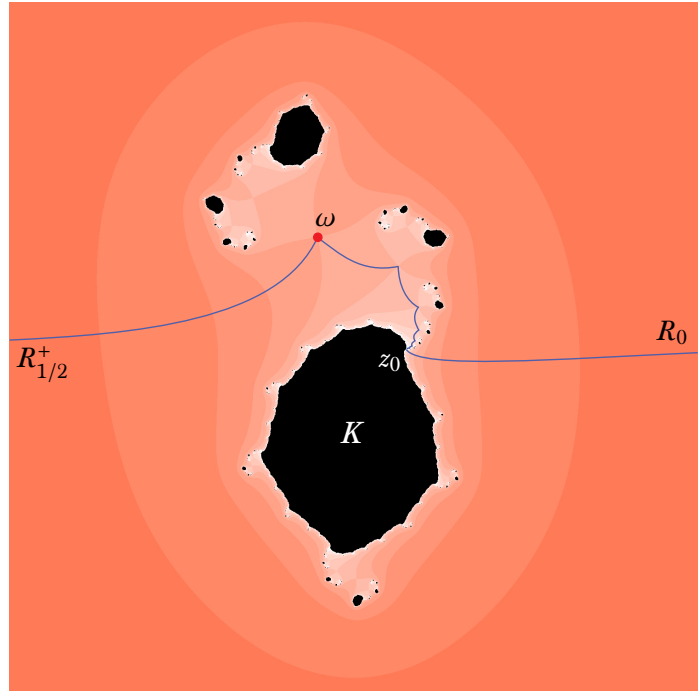

\section{Preliminaries}\label{sec:prelim}

We assume that the reader is familiar with basic complex dynamics, as in \cite{M}. Most of the following material on generalized rays can be found in greater detail in \cite{PZ}. \vs 

\noindent
{\it Convention.} For points $a,b$ on the circle $\TT=\RR/\ZZ$ we use the notation $]a,b[$ for the open interval in $\TT$ traversed counterclockwise from $a$ to $b$. This makes sense even if $a=b$ in which case $]a,a[= \TT \sm \{ a \}$.  

\subsection{The Green's function}
 
Let $P:\CC \to \CC$ be a monic polynomial map of degree $D \geq 2$. The {\bit filled Julia set} $K_P$ is the union of all bounded orbits of $P$. It is a compact non-empty subset of $\CC$ with connected complement. The complement $\CC \sm K_P$ is the {\bit basin of infinity} of $P$, that is, the set of all points which escape to $\infty$ under the iterations of $P$. The {\bit Green's function}\index{Green's function} of $P$ is the continuous function $G: \CC \to [0,+\infty[$ defined by 
$$
G(z):=\lim_{n \to \infty} \frac{1}{D^n} \log^+ |P^{\circ n}(z)|,
$$
where $\log^+ t = \max \{ \log t, 0 \}$. It measures the rate at which points escape to $\infty$ under the iterations of $P$ and satisfies the relation
$$
G(P(z))=D \, G(z) \qquad \text{for all} \ z \in \CC,
$$
with $G(z)=0$ if and only if $z \in K_P$. The Green's function is subharmonic on $\CC$ and harmonic in $\CC \sm K_P$ with a logarithmic singularity at $\infty$.\footnote{Thus, the restriction of $G$ to $\CC \sm K_P$ coincides with the Green's function of this basin in the sense of classical potential theory.} We refer to $G(z)$ as the {\bit potential} of $z$. \vs  

The critical points of $G$ outside $K_P$ are precisely the escaping (pre)critical points of $P$, that is, $\nabla G(z)=0$ for some $z \in \CC \sm K_P$ if and only if $P^{\circ n}(z)$ is a critical point of $P$ for some $n \geq 0$. To make a clear distinction between a critical point of $P$ and a critical point of $G$ in $\CC \sm K_P$, we refer to the latter as a {\bit singular point} or {\bit singularity}. It is easy to see that for every $s>0$ there are at most finitely many singularities at potentials higher than $s$. The open set $G^{-1}([0,s[)$ has finitely many connected components, all Jordan domains with piecewise analytic boundaries. \vs

The {\bit order} of a singularity $\om$ is by definition the integer $\prod_{n=0}^{\infty} \deg(P,P^{\circ n}(\omega)) \geq 2$, where $\deg(P,\cdot)$ denotes the local degree and the tail of the infinite product is $1$ since $P^{\circ n}(\omega)$ is non-critical for all large $n$.\footnote{Alternatively, $\ord(\omega)=1-\iota$, where $\iota$ is the Poincar\'e index of the vector field $\nabla G$ at its singular point $\omega$.} A singularity of order $m$ has $m$ local stable and $m$ local unstable manifolds that alternate around it and  asymptotically make an angle of $\pi/m$ to each other (see \figref{cp} left for the case $m=3$). \vs

The {\bit B\"{o}ttcher coordinate} of $P$ is the unique conformal isomorphism $B$ defined near $\infty$, normalized so that $\lim_{z \to \infty} B(z)/z = 1$, which conjugates $P$ to the power map $w \mapsto w^D$: 
$$
B(P(z))=(B(z))^D \qquad \text{for large} \ |z|.  
$$
It is related to the Green's function by the relation 
$$
\log |B(z)|=G(z) \qquad \text{for large} \ |z|. 
$$
A straightforward computation based on this relation shows that near infinity the normalized gradient $\nabla G / \| \nabla G \|^2$ is pushed forward by $B$ to the radial vector field $x \, \bd/\bd x + y \, \bd/\bd y$.  

\subsection{Smooth and broken external rays}\label{sbr}

For $\theta \in \TT$, we denote by $R_\theta$ the maximally extended smooth field line of $\nabla G/\| \nabla G \|^2$ given by $s \mapsto B^{-1}(\e^{s+2\pi i \theta})$ for large $s$. This naturally parametrizes $R_\theta$ by the potential, so for each $\theta$ there is an $s_\theta \geq 0$ such that $G(R_\theta(s)) = s$ for all $s > s_\theta$. The field line $R_\theta$ either extends all the way to the Julia set $\bd K_P$ in which case $s_\theta=0$, or it crashes into a singularity $\om$ at the potential $s_\theta>0$ in the sense that $\lim_{s \downarrow s_{\theta}} R_\theta(s)=\om$. \vs 

Here is a summary of the basic properties of the potentials $s_\theta$ (see \cite[\S 2]{PZ}): \vs

\begin{enumerate}[leftmargin=*]
\item[I.]
The function $\theta \mapsto s_\theta$ is upper semicontinuous on $\TT$. In fact, for every $s>0$ the set $\{ \theta \in \TT: s_\theta \geq s \}$ is finite. \vs 
\item[II.] 
$P(R_\theta(s))=R_{\whD(\theta)}(Ds)$ if $s>s_\theta$. As a result, we have the inequality $s_{\whD(\theta)} \leq D s_\theta$ for all $\theta \in \TT$. Equality holds if and only if $P(R_\theta(s_\theta))$ is a singular point. \vs
\item[III.]
The set $\mathcal N$ of angles $\theta \in \TT$ for which $s_\theta>0$ is countable, dense and backward-invariant under $\whD$. Moreover, ${\mathcal N} = \bigcup_{n \geq 0} \whD^{-n}({\mathcal N}_0)$, where ${\mathcal N}_0$ is the finite set of $\theta \in \TT$ for which $R_\theta$ crashes into a critical point of $P$. \vs 
\end{enumerate}

\begin{figure}[t]
\centering
\begin{overpic}[width=\textwidth]{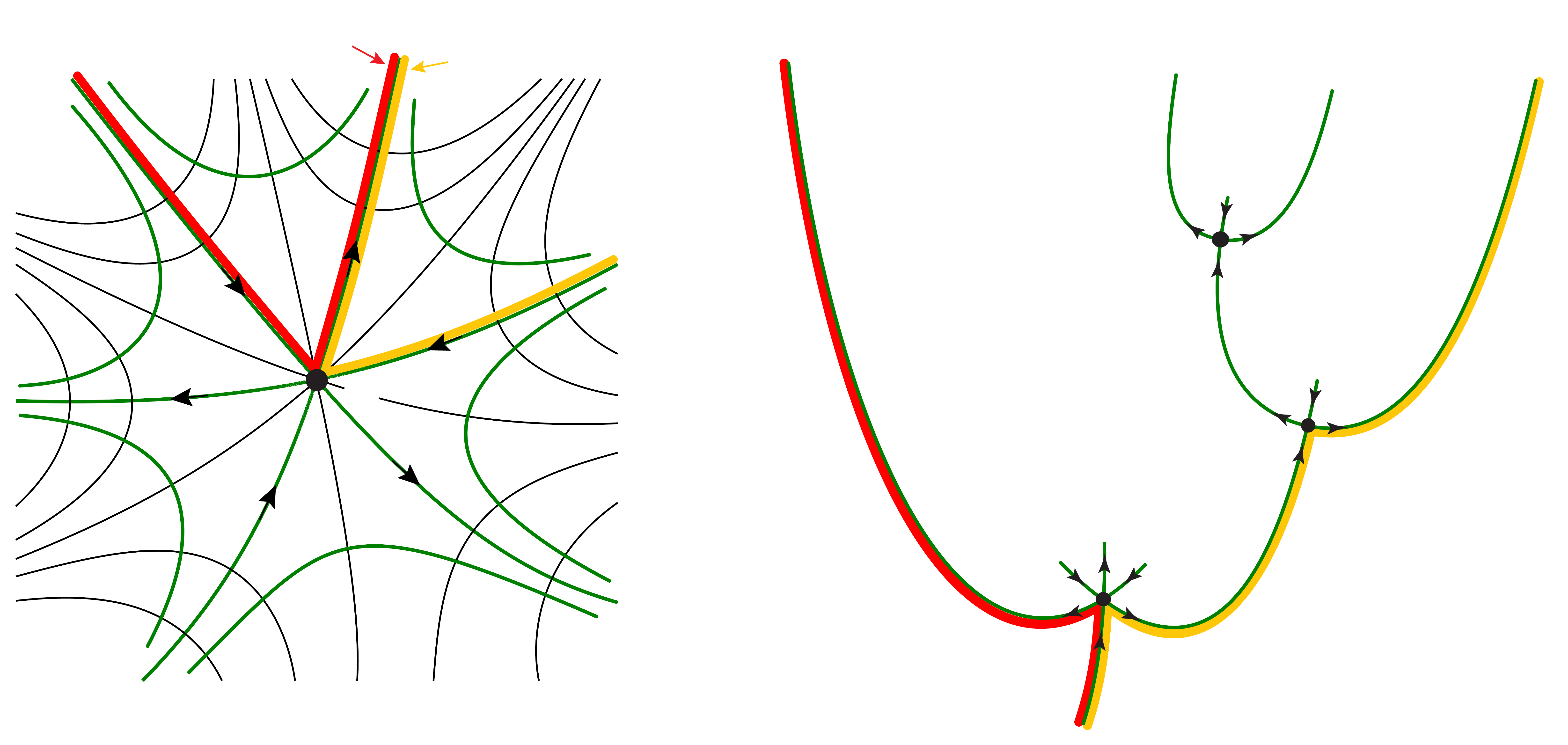}
\put (19,44) {\small $R^+_\theta$}
\put (29,43) {\small $R^-_\theta$}
\put (22.2,22) {\small $\omega$}
\put (74,30.5) {$\omega_1$}
\put (84,17.5) {$\omega_2$}
\put (71,5.5) {$\omega_3$}
\put (97,44) {\small $R^-_a$}
\put (84,43) {\small $R^-_b$}
\put (74,44) {\small $R^+_c$}
\put (46,42) {\small $R^+_d$}
\end{overpic}
\caption{\sl Left: Field lines and equipotentials of the Green's function near a singularity $\omega$ of order $3$. Each of the three local unstable manifolds can be extended past $\omega$ by turning to the immediate right or left and continuing along the corresponding stable manifold. Right: Illustration of the notion of partnership at a singularity (\S \ref{partn}). Here the partners are $R^-_b,R^+_c$ at  $\omega_1$, $R^-_a,R^+_c$ at $\omega_2$ and $R^-_a,R^+_d$ at $\omega_3$.}  
\label{cp}
\end{figure}
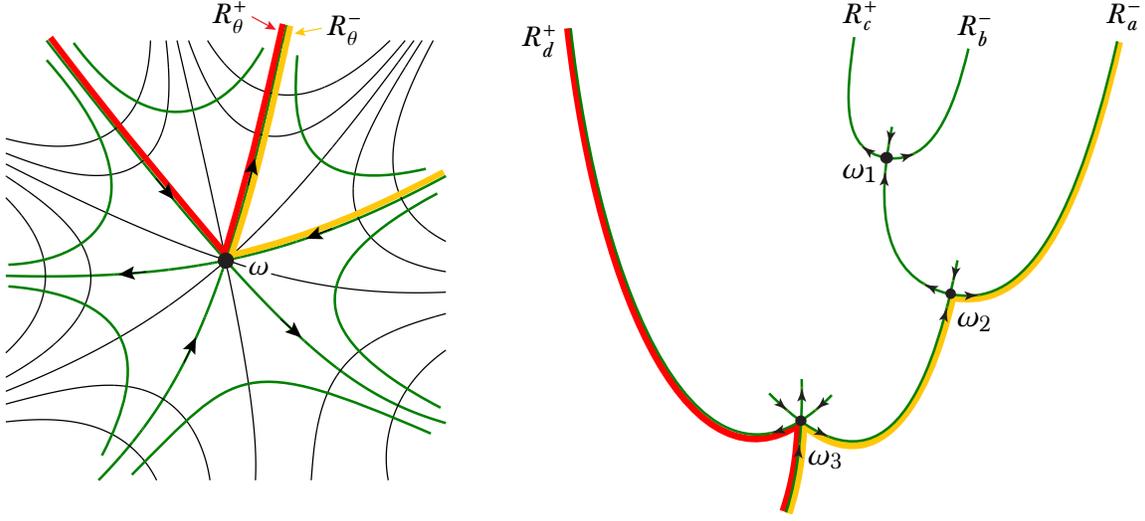

When $\theta \notin {\mathcal N}$, the field line $R_\theta$ is called the {\bit smooth ray} at angle $\theta$. When $\theta \in {\mathcal N}$, the field line $R_\theta$ is defined only for $s > s_\theta$ and there is more than one way to extend it to a curve consisting of field lines and singularities on which $G$ defines a homeomorphism onto $]0,+\infty[$. But there are always two special choices: $R_\theta^+$ which turns immediate right and $R_\theta^-$ which turns immediate left at each singularity met during the descent, i.e., moving backward in the direction of decreasing potential (\figref{cp} left). We call these extensions the {\bit right} and {\bit left broken rays} at angle $\theta$, respectively. These broken rays are also parametrized by the potential, so $R^{\pm}_{\theta}(s)$ make sense for all $s>0$, and $R^+_{\theta}(s)=R^-_{\theta}(s)=R_\theta(s)$ for $s > s_\theta$. \vs 

For potentials $s<s_\theta$ the points $R^{\pm}_\theta(s)$ belong to different connected components of $G^{-1}([0,s_\theta[)$, so the restrictions of the curves $s \mapsto R^\pm_\theta(s)$ to $]0,s_\theta[$ are disjoint. \vs

In what follows by a {\bit ray at angle} $\theta$ we mean the smooth ray $R_\theta$ when $\theta \notin {\mathcal N}$ or either of the broken rays $R^\pm_\theta$ when $\theta \in {\mathcal N}$. Each compact piece of a ray is the one-sided uniform limit of the corresponding smooth pieces of the nearby rays in the following sense: Take $\theta_0 \in \TT$ and fix $0<c<1$. By the property (I) above, $s_\theta < c$ for all $\theta$ in a deleted neighborhood of $\theta_0$. If $\theta_0 \notin {\mathcal N}$, then 
$$
\lim_{\theta \uparrow \theta_0} R_{\theta}(s) =  
\lim_{\theta \downarrow \theta_0} R_{\theta}(s) = R_{\theta_0}(s)
$$
uniformly for $s \in [c,1/c]$. On the other hand, if $\theta_0 \in {\mathcal N}$, then    
$$
\lim_{\theta \uparrow \theta_0} R_{\theta}(s) = R_{\theta_0}^-(s) \quad \text{and} \quad 
\lim_{\theta \downarrow \theta_0} R_{\theta}(s) = R_{\theta_0}^+(s)
$$
uniformly for $s \in [c,1/c]$. It follows from this and the property (II) that
\begin{align*}
P(R_\theta) & = R_{\whD(\theta)} & & \text{if} \ \theta \notin {\mathcal N} \\
P(R_\theta^\pm) & = R_{\whD(\theta)}^\pm & & \text{if} \ \theta \in {\mathcal N} \ \text{and} \ \whD(\theta) \in {\mathcal N} \\
P(R_\theta^\pm) & = R_{\whD(\theta)} & & \text{if} \ \theta \in {\mathcal N} \ \text{and} \ \whD(\theta) \notin {\mathcal N}. 
\end{align*}

\subsection{Full extensions of field lines}\label{fe}

Any local field line $\gamma$ of $\nabla G/\| \nabla G \|^2$ is contained in either a unique smooth ray or a unique pair of left and right broken rays. This can be seen by simply extending $\gamma$ in forward time (i.e., in the direction of increasing potential). Either $\gamma$ extends all the way to $\infty$ without ever hitting a singularity, or it can be extended past every singularity it meets by always turning immediate right or always turning immediate left. After finitely many such choices we eventually reach the smooth part of a field line which can be further extended to $\infty$. In the first case there is a unique $\theta$ such that $\gamma \subset R_\theta$ if $\theta \notin {\mathcal N}$ and $\gamma \subset R^-_\theta \cap R^+_\theta$ if $\theta \in {\mathcal N}$. In the second case there are unique and distinct $\theta_1, \theta_2 \in {\mathcal N}$ such that $\gamma \subset R^-_{\theta_1} \cap R^+_{\theta_2}$. We call these rays the {\bit full extensions} of $\gamma$. 

\subsection{Partnership at a singularity}\label{partn}

Let $\omega$ be a singular point at potential $s:=G(\omega)$. Two rays $R^-_a,R^+_b$ with $R^-_a(s)=R^+_b(s)=\omega$ are said to be {\bit partners at} $\omega$ if they are disjoint for potentials in $]s,+\infty[$ but coincide on $[s-\ve,s]$ for some $\ve>0$. In other words, $R^-_a,R^+_b$ must crash into $\omega$ along different unstable manifolds but bounce off along the same stable manifold (see \figref{cp} right). \vs

Suppose that the singular point $\omega$ has order $m \geq 2$ and $\gamma$ is one of the $m$ local stable manifolds at $\omega$. Let $R^-_a, R^+_b$ be the full extensions of $\gamma$ as defined in \S \ref{fe}. It is easy to see that $R^-_a, R^+_b$ are partners at $\omega$ and one of the following must occur: \vs

\begin{enumerate}[leftmargin=*]
\item[$\bullet$] 
$P(\omega)$ is non-singular and therefore contained in a unique smooth ray $R_{a'}$ or a unique pair $R^-_{a'}, R^+_{b'}$ (allowing the possibility $a'=b'$). In the first case $P(R^-_a)=P(R^+_b)=R_{a'}$, hence $\whD(a)=\whD(b)=a'$. In the second case $P(R^-_a)=R^-_{a'}, P(R^+_b)=R^+_{b'}$, hence $\whD(a)=a', \whD(b)=b'$. \vs 

\item[$\bullet$] 
$P(\omega)$ is singular of order $m/\deg(P,\omega)$. Then $\gamma':=P(\gamma)$ is a local stable manifold at $P(\omega)$. If $R^-_{a'}, R^+_{b'}$ denote the full extensions of $\gamma'$, then the rays $R^-_{a'}, R^+_{b'}$ are partners at $P(\omega)$ and $P(R^-_a)=R^-_{a'}, P(R^+_b)=R^+_{b'}$, hence $\whD(a)=a'$ and $\whD(b)=b'$.  
\end{enumerate}
 
\subsection{Periodic points and rays}\label{ppee}

Suppose $\theta$ is periodic with period $q \geq 1$ under $\whD$. Then $P^{\circ q}(R_\theta)=R_\theta$ or $P^{\circ q}(R^{\pm}_\theta)=R^{\pm}_\theta$ according as $\theta \notin {\mathcal N}$ or $\theta \in {\mathcal N}$. In the latter case, it follows from the relation 
$$
P^{\circ q}(R^{\pm}_\theta(s))=R^{\pm}_\theta(D^q s)
$$
for $s>0$ that the broken ray $R^{\pm}_\theta$ contains the infinite sequence $\{ R^{\pm}_\theta(s_\theta/D^{nq}) \}_{n \geq 0}$ of singular points and therefore is infinitely broken. If another ray at angle $\theta'$ meets $R^{\pm}_\theta$ at some point $z$, then a ray at angle $\whD^{\circ nq}(\theta')$ meets $P^{\circ nq}(R^{\pm}_{\theta})=R^{\pm}_{\theta}$ at $P^{\circ nq}(z)$ whose potential $D^{nq}G(z)$ is greater than $s_{\theta}$ if $n$ is sufficiently large. This implies $\whD^{\circ nq}(\theta')=\theta$ and proves that $\theta'$ is strictly pre-periodic. As a byproduct of this observation, we conclude that {\it distinct periodic rays are always disjoint}. \vs

For $z_0 \in K_P$, let $\La(z_0)$ denote the set of $\theta \in \TT$ for which a ray at angle $\theta$ lands at $z_0$.

\begin{theorem}\label{perr}
\mbox{}
\begin{enumerate}[leftmargin=*]
\item[(i)]
If $\theta$ has period $q$ under $\whD$, then $\theta \in \La(z_0)$ for some repelling or parabolic point $z_0 \in K_P$ whose period divides $q$. \vs
\item[(ii)]
If $z_0 \in K_P$ is a repelling or parabolic point with period $k$ under $P$, then $\La(z_0)$ is non-empty. Moreover, if the component $K$ of $K_P$ containing $z_0$ is non-degenerate, then $\La(z_0)$ is finite and all its elements have the same period under $\whD$ which is a multiple of $k$.    
\end{enumerate}
\end{theorem}

Part (i) follows from a classical application of hyperbolic metrics introduced
by Sullivan, Douady and Hubbard and is similar to the case of connected Julia sets, as in \cite[Theorem 18.10]{M}. For part (ii), see \cite[Theorem 1]{LP} and compare \cite[Theorem 6.5]{PZ}. \vs

It follows from the above remarks that each cycle of rays landing at a periodic point consists either entirely of smooth rays, or entirely of infinitely broken rays.  

\subsection{Rays accumulating on a component of $K_P$} 

Let $K$ be a {\it non-degenerate} connected component of the filled Julia set $K_P$ which is periodic with period $k \geq 1$. There are topological disks $U_1,U_0$ containing $K$ such that the restriction $P^{\circ k}|_{U_1}: U_1 \to U_0$ is a polynomial-like map of some degree $2 \leq d<D$ with connected filled Julia set $K$. Here $d$ is one more than the number of critical points of $P$ in $K$ counting multiplicities. The restriction $P^{\circ k}|_{U_1}$ is hybrid equivalent to a polynomial $\tilde{P}$ of degree $d$, that is, there is a quasiconformal map $\varphi : U_0 \to \varphi(U_0)$ which satisfies $\varphi \circ P^{\circ k} = \tilde{P} \circ \varphi$ in $U_1$ and has the property that $\bar{\bd} \varphi=0$ a.e. on $K$ \cite{DH}. It follows that $\varphi(K)$ is the filled Julia set $K_{\tilde{P}}$. The polynomial $\tilde{P}$ is unique up to affine conjugation. \vs

Let $I=I_K$ be the set of $\theta \in \TT$ for which a ray at angle $\theta$ accumulates on the component $K$. We will denote this ray by $R^K_\theta$ (which is precisely one of $R_\theta$, $R^-_\theta$, or $R^+_\theta$). The smooth external rays of $\tilde{P}$ will be denoted by $\tilde{R}_\theta$. \vs

The following is a distilled combination of two theorems in \cite{PZ}:     

\begin{theorem}\label{A}
\mbox{}
\begin{enumerate}[leftmargin=*]
\item[(i)]	
The set $I$ is compact, invariant under $\whD^{\circ k}$, and of Hausdorff dimension $<1$. Moreover, there is an essentially unique continuous degree $1$ monotone surjection $\Pi: I \to \TT$ which makes the following diagram commute:
$$
	\begin{tikzcd}[column sep=small]
	I \arrow[d,swap,"\Pi"] \arrow[rr,"\whD^{\circ k}"] & & I \arrow[d,"\Pi"] \\
	\TT \arrow[rr,"\whd"] & & \TT 
	\end{tikzcd} 
$$
\item[(ii)]
For any hybrid conjugacy $\varphi : U_0 \to \varphi(U_0)$ between the restriction $P^{\circ k}|_{U_1}: U_1 \to U_0$ and a degree $d$ monic polynomial $\tilde{P}$ there is a choice of the semiconjugacy $\Pi$ as in (i) such that 
$R^K_\theta$ and $\varphi^{-1}(\tilde{R}_{\Pi(\theta)})$ have the same accumulation set for all $\theta \in I$. In particular, $R^K_\theta$ lands at $z \in K$ if and only if $\tilde{R}_{\Pi(\theta)}$ lands at $\varphi(z) \in K_{\tilde{P}}$.
\end{enumerate}
\end{theorem}

Observe that $I$ is nowhere dense since it has Hausdorff dimension $<1$ and it is uncountable since it maps continuously onto $\TT$. However, $I$ may have isolated points, so in general it is not a Cantor set \cite{PZ}. \vs 

For $s>0$ let $V_s$ be the connected component of the open set $G^{-1}([0,s[)$ that contains $K$. Then $\{ V_s \}_{s>0}$ is a properly nested collection of Jordan domains (in the sense that $\ov{V_s} \subset V_{s'}$ whenever $s<s'$) and $K=\bigcap_{s>0} V_s$. Thus, $\theta \in I$ if and only if a ray at angle $\theta$ meets $V_s$ for every $s>0$. \vs

We will need the following simple observation in the proof of the Main Theorem: 
  
\begin{lemma}\label{wake}
Suppose $R^-_a,R^+_b$ are partners at a singularity $\omega$ in the sense of \S \ref{partn}. If $R^-_a$ or $R^+_b$ accumulates on $K$, then $]a,b[ \cap I = \es$. 
\end{lemma}

\begin{proof}
Let $s:=G(\omega)$ and $W$ be the component of $\CC \sm (R^-_a([s,+\infty[) \cup R^+_b([s,+\infty[))$ not containing $K$. Then all rays at angles in $]a,b[$ are contained in the closure of $W$, hence none can accumulate on $K$.     
\end{proof}

\section{The Proof}\label{sec:proof}

The proof of the Main Theorem begins as follows. After passing to a high enough iterate of $P$ we may assume that $z_0$ is a fixed point of $P$. Since $K \neq \{ z_0 \}$, \thmref{perr}(ii) guarantees that $\La(z_0)$ is non-empty and finite, and all angles in $\La(z_0)$ have a common period $q$ under $\whD$. By passing to the $q$-th iterate of $P$ and replacing $D^q$ with $D$, we can reduce to the case where all angles in $\La(z_0)$ are fixed under $\whD$ and therefore belong to the finite set $\{ j/(D-1) \modd: 0 \leq j \leq D-2 \}$.  \vs

Take any $a_1 \in \La(z_0)$. There is nothing to prove if $R^K_{a_1}=R_{a_1}$ is smooth, so let us assume $R^K_{a_1}=R^-_{a_1}$ (the case where $R^K_{a_1}=R^+_{a_1}$ is similar). Consider the critical point $\omega_1:=R^-_{a_1}(s_{a_1})$ of highest potential on $R^-_{a_1}$. Since $\whD(a_1)=a_1$, the ray $R^-_{a_1}$ is infinitely broken at the preimages $\omega_n := R^-_{a_1}(s_{a_1}/D^{n-1})$ of $\omega_1$. For $n \geq 1$ let $R^+_{b_n}$ be the partner of $R^-_{a_1}$ at $\omega_n$ as defined in \S \ref{partn}. Choosing representative angles between $a_1$ and $a_1+1$, we evidently have  $b_0:=a_1<b_1<b_2<b_3<\cdots<a_1+1$ (\figref{wakes}). Let us check that $\whD(b_n)=b_{n-1}$ for all $n \geq 1$. First note that the critical value $P(\omega_1)$ is non-singular and belongs to both $R^-_{a_1}, R^+_{a_1}$. Since $P(R^-_{a_1})=R^-_{a_1}$, we have $P(R^+_{b_1})=R^+_{a_1}$ and therefore $\whD(b_1)=a_1=b_0$ (see the end of \S \ref{partn}). For $n>1$, since $P(R^-_{a_1})=R^-_{a_1}$, it follows that $P(R^+_{b_n})$ is the partner of $R^-_{a_1}$ at $\omega_{n-1}$, hence $P(R^+_{b_n})=R^+_{b_{n-1}}$, which gives  $\whD(b_n)=b_{n-1}$, as required.\footnote{In fact, one can show that $b_{n+1}-b_n=(b_n-b_{n-1})/D$ for all large enough $n$, but we will not need this for our purposes.} \vs  

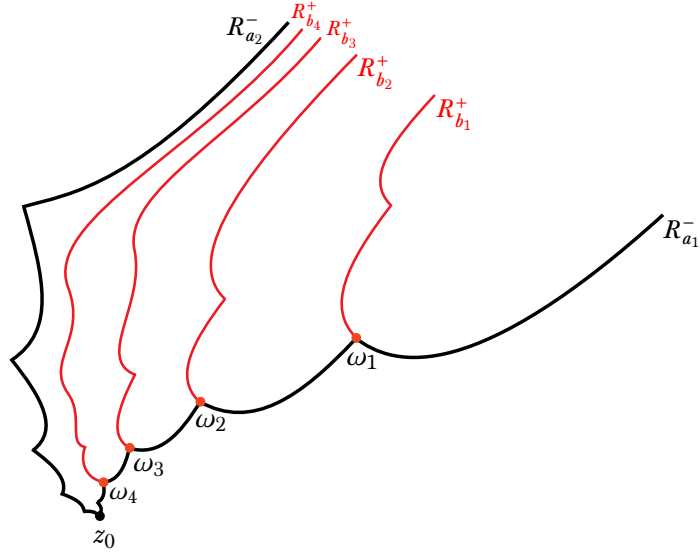
\begin{figure}[t]
	\centering
	\begin{overpic}[width=0.6\textwidth]{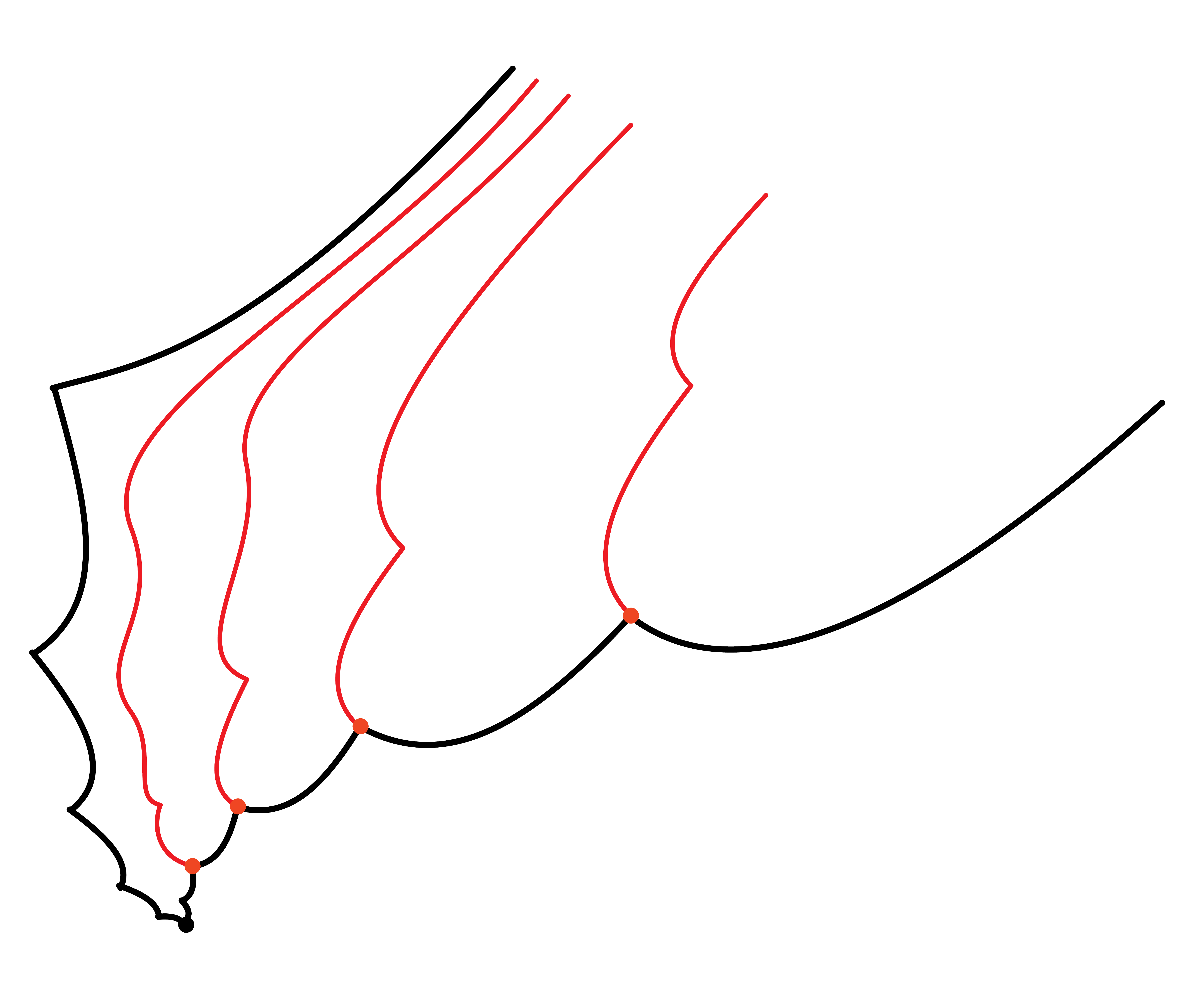}
		\put (97.5,47.5) {\footnotesize $R^-_{a_1}$}
		\put (64.5,65) {\color{red}{\footnotesize $R^+_{b_1}$}}
		\put (53,71) {\color{red}{\footnotesize $R^+_{b_2}$}}
		\put (48.5,77) {\color{red}{\tiny $R^+_{b_3}$}}
		\put (43.4,79.5) {\color{red}{\tiny $R^+_{b_4}$}}
		\put (34,77) {\footnotesize $R^-_{a_2}$}
		\put (51.5,29) {\small $\omega_1$}
		\put (29.5,20) {\small $\omega_2$}
		\put (20.5,13.5) {\small $\omega_3$}
		\put (17,9.5) {\footnotesize $\omega_4$}
		\put (14.5,3) {\small $z_0$}
	\end{overpic}
	\caption{\sl The successive $\whD$-preimages of the interval $]a_1,b_1[$ exhaust the interval $]a_1,a_2[$ between two fixed points of $\whD$. The partners $R^-_{a_1}, R^+_{b_n}$ at $\omega_n$ block rays with angles in $]a_1,b_n[$ from accumulating on $K$.}  
	\label{wakes}
\end{figure}

Now let $a_2$ be the limit of the monotone sequence $\{ b_n \}$. Letting $n \to \infty$ in the relation $\whD(b_n)=b_{n-1}$ shows that $a_2$ is a fixed point of $\whD$. \vs

Recall that $I=I_K$ is the set of $\theta \in \TT$ for which a ray at angle $\theta$ accumulates on $K$. In particular, $I \supset \La(z_0)$.  
 
\begin{lemma}\label{gap}
$]a_1 ,a_2[ \cap I = \es$. In particular, $a_2 \neq a_1 \modd$. 	
\end{lemma}

\begin{proof}
For each $n \geq 1$ the rays $R^-_{a_1}, R^+_{b_n}$ are partners at $\omega_n$, so $]a_1,b_n[ \cap I = \es$ by \lemref{wake}. Since $]a_1,a_2[ = \bigcup_{n\geq 1} ]a_1,b_n[$, we conclude that $]a_1 ,a_2[ \cap I = \es$. If $a_2 = a_1 \modd$, the set $I$ would consist of a single point, which would be impossible since $I$ is uncountable. 
\end{proof} 

\begin{lemma}\label{next}
$a_2 \in \La(z_0)$. 
\end{lemma} 

The proof will also show that $R^K_{a_2}=R_{a_2}$ or $R^-_{a_2}$ according as $a_2 \notin {\mathcal N}$ or $a_2 \in {\mathcal N}$. 

\begin{proof}
It suffices to verify that $a_2 \in I$ for then \thmref{A}(i) implies $\Pi(a_1)=\Pi(a_2)$ by \lemref{gap} and monotonicity of $\Pi$, and \thmref{A}(ii) implies that the rays $R^K_{a_1}=R^-_{a_1}$ and $R^K_{a_2}$ have the same accumulation set, namely $\{ z_0 \}$. \vs
 
Recall that $V_s$ is the component of $G^{-1}([0,s[)$ containing $K$. Consider the potentials $\si_n:=s_{a_1}/D^{n-1}$. Take an arbitrary integer $k \geq 1$ and a constant $0<c<1$ such that $c<\si_k<1/c$. The singularity $\omega_n =R^-_{a_1}(\si_n)=R^+_{b_n}(\si_n)$ is on the boundary of $V_{\si_n}$, so $R^+_{b_n}(\si_n) \in V_{\si_k}$ for all $n>k$. This, in turn, implies that $R^+_{b_n}(\si_k) \in \bd V_{\si_k}$ for all $n>k$. If $a_2 \notin 
{\mathcal N}$, then $\lim_{n \to \infty} R_{b_n}(s)=R_{a_2}(s)$ uniformly for $c \leq s \leq 1/c$ (see the end of \S \ref{sbr}), hence $R_{a_2}(\si_k) \in \bd V_{\si_k}$. Since $k$ was arbitrary, we conclude that $R_{a_2}$ enters $V_s$ for all $s>0$, hence $a_2 \in I$. On the other hand, if $a_2 \in {\mathcal N}$, then $\lim_{n \to \infty} R_{b_n}(s)=R^-_{a_2}(s)$ uniformly for $c \leq s \leq 1/c$ and a similar reasoning shows that $R^-_{a_2}$ enters $V_s$ for all $s>0$, so again $a_2 \in I$.   
\end{proof}

Now if $R^K_{a_2}=R_{a_2}$ is smooth, we are done. Otherwise, by the proof of \lemref{next} $R^K_{a_2}=R^-_{a_2}$ and we can repeat the above argument for $a_2$ in place of $a_1$. Inductively, assuming $a_{n-1} \in \La(z_0)$ and $R^K_{a_{n-1}}=R^-_{a_{n-1}}$ is broken, we can construct the next $a_n \in \La(z_0)$ with the property $]a_{n-1},a_n[ \cap I = \es$. This $a_n$ must be distinct from all its predecessors $a_1, \ldots, a_{n-1}$ since otherwise $I$ would be finite. As every $a_n$ is a fixed point of $\whD$, this process terminates in at most $D-1$ steps. The cause of termination must be reaching an $a_n$ for which $R^K_{a_n}=R_{a_n}$ is smooth.

\end{document}